\documentclass[11pt,a4paper]{amsart}
\usepackage{amsthm,amsmath}
\usepackage[all]{xy}
\title{A structure theorem for unions of complete intersections} 
\author{Alfio Ragusa
	\and  Giuseppe Zappal\`a}

\subjclass[2000]{13 D 40, 13 H 10}

\DeclareSymbolFont{rsfscript}{OMS}{rsfs}{m}{n}
\DeclareSymbolFontAlphabet{\mathrsfs}{rsfscript}

\DeclareSymbolFont{AMSb}{U}{msb}{m}{n}
\DeclareSymbolFontAlphabet{\mathbb}{AMSb}

\DeclareSymbolFont{eufrak}{U}{euf}{m}{n}
\DeclareSymbolFontAlphabet{\gothic}{eufrak}

\def\ac{\`}

\newcommand\sgn{\operatorname{sgn}}
\newcommand{\f}{\footnotesize}

\newcommand{\pp}{\mathbb P}
\newcommand\hgt{\operatorname{ht}}

\newcommand\pf{\operatorname{pf}}
\newcommand\Pf{\operatorname{Pf}}

\newcommand\rw{\Rightarrow}

\newcommand\zz{{\mathbb Z}}

\newtheorem{thm}{Theorem}[section]
\newtheorem{lem}[thm]{Lemma}
\newtheorem{prp}[thm]{Proposition}
\newtheorem{cor}[thm]{Corollary}

\theoremstyle{definition}

\theoremstyle{remark}
\newtheorem{rem}[thm]{Remark}
\newtheorem{exm}[thm]{Example}

\begin{document}


\begin{abstract} Using the connections among almost complete intersection schemes, arithmetically Gorenstein schemes and schemes that are union of complete intersections we give a structure theorem for arithmetically Cohen-Macaulay union of two complete intersections of codimension $2.$ We apply the results for computing Hilbert functions and graded Betti numbers for such schemes.

\end{abstract}

\maketitle

\section*{Introduction}
\markboth{\it Introduction}{\it Introduction}
The simplest projective schemes which one can study are those whose defining ideals can be generated by the minimal number of equations with respect to their codimension $c.$ These schemes are the complete intersection schemes whose defining ideals are minimally generated by exactly $c$ elements.
In particular, this implies that they can be generated by a regular sequence. Now, there are different ways to generalize such a notion. For instance, one can ask to study arithmetically Cohen Macaulay schemes which can be generated by a number of elements which is equal to one more than the codimension. These kinds of schemes are usually denominated almost complete intersections and they are recently studied, for instance, in \cite{MMN}, \cite{HK}, \cite{MM}, \cite{RZ3} and \cite{Se}. Another possibility is to study schemes which, as the complete intersections,  have Cohen-Macaulay type $1$ or equivalently with principal last syzygies module. In this case we have arithmetically Gorenstein schemes and we have a large literature on this theme (see, for instance, \cite{BE}, \cite{Di}, \cite{IK}, \cite{RZ1}, \cite{RZ2}). Finally, from a more geometric point of view, one can ask to study schemes which are a finite union of complete intersections with some kind of generic property for realizing such unions. Well, all these kinds of generalizations are strictly related, as we show in this paper in the codimension two case. Indeed, using information on an almost complete intersection and on a Gorenstein scheme directly linked to it we obtain nice information on a union of two complete intersection schemes. The idea is very simple: starting from a union of two complete intersections $X_1$ and $X_2$ of codimension $2$ in $\pp^r$ which realizes an arithmetically Cohen-Macaulay scheme $X$ (this, in particular, implies that $X_1\cap X_2$ is aCM of codimension $3$), using their defining ideals $I_{X_{1}},$ $I_{X_{2}}$ in the polynomial ring $R=k[x_0, \ldots, x_r]$, one performs the almost complete intersection $I_{X_{1}}+I_{X_{2}}$  of codimension $3$ which, when it is linked in a complete intersection generated by three of its generators, produces a Gorenstein scheme $G$ of codimension $3.$ Now, we use the pfaffian resolution of this Gorenstein scheme to obtain a free resolution of the given union $I_X$ getting in this way a structure theorem for such schemes (Theorem \ref{matr}). As applications of this result we obtain a description of the Hilbert functions of these schemes, in particular, the starting degree and the Castelnuovo-Mumford regularity for $I_X$, and a lot of information about their graded Betti numbers (indeed, all up to very few cancellations). In many cases our resolutions are minimal, so in these cases we get the Hilbert-Burch matrix of the defining ideal of these schemes.

\section{Union of complete intersections of codimension two} 
\markboth{\it Union of complete intersections of codimension two}
{\it Union of complete intersections of codimension two}
Let $k$ be an algebraically closed field and let $R:=k[x_0,\ldots,x_r].$ We consider on $R$ the standard grading and we consider in it just homogeneous ideals.\par 
An ideal $I_Q\subset R$ is said to be an {\em almost complete intersection} ideal of codimension $c$ if $I_Q$ is perfect and it is minimally generated by less or equal to $c+1$ forms (note that we include complete intersections in this definition).\par
Every almost complete intersection ideal $I_Q$ of codimension $c$ is directly linked in a complete intersection to a Gorenstein ideal $I_G\subset R.$ Indeed, if $I_Z\subseteq I_Q$ is generated by $c$ minimal generators of $I_Q,$ which form a regular sequence, then $I_G:=I_Z:I_Q$ is a Gorenstein ideal. By liaison theory (see \cite{PS} for a complete discussion on this argument) we have also $I_Q=I_Z:I_G.$ \par
Gorenstein and almost complete intersection ideals in codimension $3$ were extensively studied. In particular, it is well known that the Gorenstein ideals of codimension three are generated by the $(n-1)$-pfaffians of an alternating matrix of odd size $n$ (see \cite{BE}). If $A$ is an alternating matrix we will denote by $\pf A$ its pfaffian and by $\Pf_r(A)$ the ideal generated by the $r$-pfaffians of $A.$
\par
Let $X_1,X_2\subset\pp^r,$ $r\ge 2$ be two complete intersection schemes of codimension $2$ without common components. Assume that $X_1\cup X_2$ is aCM. 
For instance, note that this always happens for a disjoint union of two $0$-dimensional complete intersection schemes of $\pp^2.$ 

\begin{rem}
By the standard exact sequence
\begin{equation}\label{vietoris}
 0\to I_{X_1}\cap I_{X_2}\to I_{X_1}\oplus I_{X_2}\to I_{X_1}+I_{X_2}\to 0
\end{equation}
we see that the homological dimension of $R/(I_{X_1}+ I_{X_2})$ is less than or equal to $3$ (by mapping cone), hence, since by assumption, $X_1,X_2$ have no common components, it is exactly $3.$ Consequently, the ideal $ I_{X_1}+I_{X_2}$ is Artinian for $r=2$ and it is a saturated ideal for $r\ge 3;$ precisely, $ I_{X_1}+I_{X_2}=I_{X_1\cap X_2},$ therefore $X_1\cap X_2$ is aCM of codimension $3.$
\end{rem}

Now we would like to give a structure theorem for schemes of the type $X_1\cup X_2.$

We start with collecting some well known facts about pfaffians and determinants of skew-symmetric matrices which will be useful for proving our results.
\par
We introduce the following notation. If $M$ is a matrix we will denote by $M_{[i_1,\ldots,i_r;j_1,\ldots,j_s]}$ the submatrix of $M$ obtained by deleting the rows labelled by
the integers $i_1,\ldots,i_r$ and the columns labelled by the integers $j_1,\ldots,j_s.$ With $M_{[i_1,\ldots,i_r;-]}$ we will denote the submatrix of $M$ obtained by deleting just the rows, and analogously for $M_{[-;j_1,\ldots,j_s]}.$ Moreover $M_{[i_1,\ldots,i_r;i_1,\ldots,i_r]}$ will be denoted by $M_{(i_1,\ldots,i_r)}.$

In the sequel if $a_1<\ldots<a_n$ are integers and $(b_1,\ldots,b_n)$ is a permutation of them, we will use $\sgn(b_1,\ldots,b_n)$ the sign of the permutation $(b_1,\ldots,b_n)$ with respect to $(a_1,\ldots,a_n).$

\begin{lem}\label{detpf}
Let $A$ be an alternating matrix of odd size $n.$ Then
\begin{itemize}
	\item[1)] $\det A_{[i;j]}=\pf A_{(i)}\pf A_{(j)}.$
  \item[2)] $\det A_{[i,j;h,k]}=\sgn(i,j)\sgn(h,k)\left(\sgn(i,j,h)\pf A_{(i,j,h)}\pf A_{(k)}-\right.$ \\
  $\left.\sgn(i,j,k)\pf A_{(i,j,k)}\pf A_{(h)}\right)=$\\
  $\sgn(h,k)\sgn(i,j)\left(\sgn(h,k,j)\pf A_{(h,k,j)}\pf A_{(i)}-\right.$ \\
  $\left.\sgn(h,k,i)\pf A_{(h,k,i)}\pf A_{(j)}\right).$
\end{itemize}
\end{lem}
\begin{proof}
The first result is due to Cayley (see \cite{Cay}). The second one is a rewriting of a generalization due to Heymans (see \cite{Hey}, formula (3.31)).
\end{proof}

\begin{lem}\label{genpf}
Let $A$ be an alternating matrix of odd size $n$ and let $p_i:=(-1)^i\pf A_{(i)}.$ Let 
$$B:=\begin{pmatrix}
 0&a_1&\ldots&a_n\\
 b_1\\
 \vdots& & \text{{\huge A}}\\
 b_n
 \end{pmatrix}.$$
Then $$\det B=\sum_{i=1}^na_ip_i\sum_{j=1}^nb_jp_j\,.$$
\end{lem}
\begin{proof}
Using Lemma \ref{detpf} part 1), we know that 
 $$\det A_{[i;j]}=\pf A_{(i)}\pf A_{(j)}.$$ 
So to get the result it is enough to compute $\det B$ by using Laplace rule with respect to the first row and the first column.
\end{proof}

\begin{thm}\label{matr}
\phantom{A}
\par
1) Let $X_1,X_2\subset\pp^r,$ $r\ge 2,$ be two complete intersection schemes of codimension $2$ of type, respectively, $(d_1,e_1)$ and $(d_2,e_2),$ such that $\min\{d_1,e_1\}>\min\{d_2,e_2\},$  without common components. Assume that $X:= X_1\cup X_2$ is aCM.
Then there exists an alternating matrix $A$ of odd size $n$ with entries in $R$ and $3n$ forms $\alpha_i,$ $\beta_i,$ $\gamma_i,$ $1\le i\le n,$ such that
 $I_X$ is the ideal generated by the maximal minors of the matrix
 $$M=\begin{pmatrix}0&\beta_1&\beta_2&\ldots&\beta_{n-1}&\beta_n\\
 0&\alpha_1&\alpha_2&\ldots&\alpha_{n-1}&\alpha_n\\
 \gamma_1\\
 \gamma_2\\
 \vdots& & & \text{{\huge A}}\\
 \gamma_{n-1}\\
 \gamma_n
 \end{pmatrix}.$$
Precisely, it is possible to choose four forms $f_1,g_1,f_2,g_2,$ with $g_2$ of minimal degree among the four forms, such that $I_{X_1}=(f_1,g_1),$ $I_{X_2}=(f_2,g_2),$ $(f_1,g_1,f_2)$ is a regular sequence and we can choose $M$ in such a way that 
$f_1=\sum_{i=1}^n\alpha_ip_i,$ $g_1=\sum_{i=1}^n\beta_ip_i,$ $f_2=\sum_{i=1}^n\gamma_ip_i$ and $g_2=\pf\overline{A},$ where $p_i=(-1)^i\pf A_{(i)}$ and 
$$\overline{A}=\begin{pmatrix}0&0&0&\gamma_1&\ldots&\gamma_n\\0&0&0&\beta_1&\ldots&\beta_n\\0&0&0&\alpha_1&\ldots&\alpha_n\\
 -\gamma_1&-\beta_1&-\alpha_1\\
 \vdots&\vdots&\vdots&&\text{\huge A}\\
 -\gamma_n&-\beta_n&-\alpha_n
 \end{pmatrix}.$$
2) If $M$ is a matrix as above where $\Pf_{n-1}A$ is an ideal of height $3,$ we set $f_1:=\sum_{i=1}^n\alpha_ip_i,$ $g_1:=\sum_{i=1}^n\beta_ip_i,$ $f_2:=\sum_{i=1}^n\gamma_ip_i,$ where the $p_i$'s are as in 1). If $(f_1,g_1,f_2)$ is a regular sequence and $f_2$ and $\pf\overline{A}$ are coprime
then the ideal generated by the maximal minors of $M$ defines a scheme which is a union of two complete intersections of codimension $2.$
\end{thm}
\begin{proof}
1) Let $I_{X_1}=(f_1,g_1)$ and $I_{X_2}=(f_2,g_2),$ with $g_2$ having minimal degree with respect the four forms. Consequently we can choose $f_1,g_1,f_2$ such that they form a regular sequence (this can be done since the codimension of $X_1\cap X_2$ is $3$). Denote by $I_Q=I_{X_1}+I_{X_2}$
and $I_Z=(f_1,g_1,f_2).$ Let $I_G:=I_Z:I_Q$ and note that since $I_Q$ is the ideal of an almost complete intersection and $I_Z$ is the ideal of a complete intersection contained in it, $I_G$ is the ideal of a Gorenstein scheme of codimension $3.$ By the structure theorem of Buchsbaum and Eisenbud (see \cite{BE}), there exists an alternating matrix $A$ of odd size $n,$ such that $I_G=\Pf_{n-1}(A).$ So $I_G=(p_1,\ldots,p_n),$ where $p_i=(-1)^i\pf A_{(i)}.$ Since $f_1,g_1,f_2\in I_G,$ we can write $f_1=\sum_{i=1}^n\alpha_ip_i,$ $g_1=\sum_{i=1}^n\beta_ip_i,$ $f_2=\sum_{i=1}^n\gamma_ip_i.$ Furthermore, in this setting, since $\min\{d_1,e_1\}>\min\{d_2,e_2\},$ using results in \cite{RZ3} we have that $(f_2,g_2)=(f_2,\pf\overline{A})$ (in particular if $e_2<d_2$ then $g_2=\pf\overline{A},$ up to a unit). Then we set
$$M:=\begin{pmatrix}0&\beta_1&\beta_2&\ldots&\beta_{n-1}&\beta_n\\
 0&\alpha_1&\alpha_2&\ldots&\alpha_{n-1}&\alpha_n\\
 \gamma_1\\
 \gamma_2\\
 \vdots& & & \text{{\huge A}}\\
 \gamma_{n-1}\\
 \gamma_n
 \end{pmatrix}.$$
We want to show that $I_X=I_{n+1}(M).$ At first we show that $$I_{n+1}(M)\subseteq I_X=I_{X_1}\cap I_{X_2}.$$
By Lemma \ref{genpf}, we get immediately that 
 $$\det M_{[1;-]}=f_1f_2,\,\,\det M_{[2;-]}=g_1f_2.$$
Let now $t$ be an integer, $3\le t\le n+2.$ We want to compute $\det M_{[t;-]}.$ To do that we apply Laplace rule with respect to the first row, the second row and the first column. So if we set
 $$\sigma_{ij}:=\left\{\begin{array}{lll} (-1)^{i+j}& \text{ if } & i<j \\(-1)^{i+j+1}& \text{ if } & i>j \end{array}\right.,\,
   \tau_{ij}:=\left\{\begin{array}{lll} (-1)^{i+1}& \text{ if } & i<j \\(-1)^{i}& \text{ if } & i>j \end{array}\right.$$
and $S_t:=\{(i,j,h)\mid 1\le i,j,h\le n,\,i\ne j,\,h\ne t-2\},$ $3\le t\le n+2,$
we obtain the following expansion
$$\det M_{[t,-]}=\sum_{S_t}\sigma_{ij}\tau_{h,t-2}\alpha_i\beta_j\gamma_h\det A_{[t-2,h;i,j]}.$$
Applying Lemma \ref{detpf} part 2), we have
\begin{multline*}
\det M_{[t,-]}=\\ \sum_{S_t}\sigma_{ij}\tau_{h,t-2}\alpha_i\beta_j\gamma_h\sgn(t-2,h)\sgn(i,j)(\sgn(t-2,h,i)\pf A_{(t-2,h,i)}\pf A_{(j)}- \\ \sgn(t-2,h,j)\pf A_{(t-2,h,j)}\pf A_{(i)})=\\
\sum_{S_t}(-1)^{i+h}\sgn(t-2,h,i)\alpha_i\beta_j\gamma_h\pf A_{(t-2,h,i)}p_j-\\ \sum_{S_t}(-1)^{j+h}\sgn(t-2,h,j)\alpha_i\beta_j\gamma_h\pf A_{(t-2,h,j)}p_i= \\
\sum_j\left(\sum_{i\ne j;h\ne t-2}(-1)^{i+h}\sgn(t-2,h,i)\alpha_i\gamma_h\pf A_{(t-2,h,i)}\right)\beta_jp_j-\\ \sum_i\left(\sum_{j\ne i;h\ne t-2}(-1)^{j+h}\sgn(t-2,h,j)\beta_j\gamma_h\pf A_{(t-2,h,j)}\right)\alpha_ip_i
\end{multline*} 
now we sum up and subtract the quantity
$$\sum_{k;h\ne t-2}(-1)^{k+h}\sgn(t-2,h,k)\pf A_{(t-2,h,k)}\alpha_k\gamma_h\beta_kp_k$$

so we get
\begin{multline*}
\sum_j\left(\sum_{i;h\ne t-2}(-1)^{i+h}\sgn(t-2,h,i)\alpha_i\gamma_h\pf A_{(t-2,h,i)}\right)\beta_jp_j+\\ -\sum_i\left(\sum_{j;h\ne t-2}(-1)^{j+h}\sgn(t-2,h,j)\beta_j\gamma_h\pf A_{(t-2,h,j)}\right)\alpha_ip_i;
\end{multline*}
if we set $$\lambda:=\sum_{i;h\ne t-2}(-1)^{i+h}\sgn(t-2,h,i)\alpha_i\gamma_h\pf A_{(t-2,h,i)}$$ and 
$$\mu:=\sum_{j;h\ne t-2}(-1)^{j+h}\sgn(t-2,h,j)\beta_j\gamma_h\pf A_{(t-2,h,j)}$$
we have
 $$\lambda\sum_j\beta_jp_j-\mu\sum_i\alpha_ip_i=\lambda g_1-\mu f_1\in (f_1,g_1)=I_{X_1}.$$



On the other hand

$$\det M_{[t,-]}=\sum_{S_t}\sigma_{ij}\tau_{h,t-2}\alpha_i\beta_j\gamma_h\det A_{[t-2,h;i,j]}.$$
Applying the second equality in Lemma \ref{detpf} part 2), we have
\begin{multline*}
\det M_{[t,-]}=\\ \sum_{S_t}\sigma_{ij}\tau_{h,t-2}\alpha_i\beta_j\gamma_h\sgn(h,t-2)\sgn(i,j)(\sgn(i,j,t-2)\pf A_{(i,j,t-2)}\pf A_{(h)}+ \\ -\sgn(i,j,h)\pf A_{(i,j,h)}\pf A_{(t-2)})=\\
\sum_{S_t}(-1)^{i+j+1}\sgn(i,j,t-2)\alpha_i\beta_j\gamma_h\pf A_{(i,j,t-2)}p_h-\\ \sum_{S_t}(-1)^{i+j+h+1}\sgn(i,j,h)\alpha_i\beta_j\gamma_h\pf A_{(i,j,h)}\pf A_{(t-2)}).
\end{multline*} 
Now we sum up and subtract the quantity
$$\sum_{i,j;i\ne j}(-1)^{i+j+1}\sgn(i,j,t-2)\alpha_i\gamma_j\beta_{t-2}\pf A_{(i,j,t-2)}p_{t-2}.$$
And we get
\begin{multline*}
\sum_h\left(\sum_{i,j;i\ne j}(-1)^{i+j+1}\sgn(i,j,t-2)\alpha_i\beta_j\pf A_{(i,j,t-2)}\right)\gamma_hp_h-\\ \pf A_{(t-2)}
\sum_{i,j,h;i\ne j}(-1)^{i+j+h+1}\sgn(i,j,h)\alpha_i\beta_j\gamma_h\pf A_{(i,j,h)}=\\
\left(\sum_{i,j;i\ne j}(-1)^{i+j+1}\sgn(i,j,t-2)\alpha_i\beta_j\pf A_{(i,j,t-2)}\right)\sum_h\gamma_hp_h-\pf A_{(t-2)}\pf\overline{A}=\\
\left(\sum_{i,j;i\ne j}(-1)^{i+j+1}\sgn(i,j,t-2)\alpha_i\beta_j\pf A_{(i,j,t-2)}\right)f_2-\pf A_{(t-2)}\pf\overline{A}\\ \in (f_2,\pf\overline{A})=I_{X_2}.
\end{multline*}
note that, in the last step, we computed $\pf\overline{A}$ by the Laplace rule with respect to the $3\times 3$ minors of the first three rows. So $I_{n+1}(M)$ is contained in $I_{X_1}\cap I_{X_2}.$
\par
Let now $Y$ be the scheme defined by the ideal $I_{n+1}(M).$ To complete the proof will be enough to show that $\deg Y=\deg X=\deg X_1+\deg X_2.$ Let $d_i:=\deg f_i$ and
$e_i:=\deg g_i.$ So we have to show that $\deg Y=d_1e_1+d_2e_2.$
\par
Now from $M$ we can deduce the degrees of a set of generators and of the corresponding syzygies of $I_Y.$ The degrees of generators are $d_1+d_2,e_1+d_2,e_2+\pi_1,\ldots,e_2+\pi_n,$ where $\pi_i=\deg p_i.$ The degrees of syzygies are $e_2+d_2,d_1+e_1+d_2-\pi_1,\ldots,d_1+e_1+d_2-\pi_n.$
It is well known that
 $$2\deg Y=(e_2+d_2)^2+\sum_{i=1}^n(d_1+e_1+d_2-\pi_i)^2-(d_1+d_2)^2-(e_1+d_2)^2-\sum_{i=1}^n(e_2+\pi_i)^2.$$
Since the $p_i$'s are minimal generators for $I_G=I_Z:I_Q,$ we have that $2\sum_{i=1}^n\pi_i=(n-1)(d_1+e_1+d_2-e_2),$ see for instance \cite{RZ1}. Now a straightforward computation shows that
 $$2\deg Y=2d_1e_1+2d_2e_2$$
and we are done.
\par
2) Let $I_G:=\Pf_{n-1}A.$ $I_G$ defines an aG scheme of codimension $3.$ Of course $I_G$ contains the complete intersection ideal $I_Z:=(f_1,g_1,f_2).$ Now we set $I_Q:=I_Z:I_G.$ Using the results in \cite{RZ3} $I_Q=(f_1,g_1,f_2,g_2),$ where $g_2:=\pf\overline{A}.$ Let $I_{X_1}:=(f_1,g_1)$ and $I_{X_2}:=(f_2,g_2).$ By the hypotheses $X_1$ and $X_2$ are complete intersection schemes. By the part 1) $I_{n+1}(M)=I_{X_1\cup X_2}.$

\end{proof}

\begin{rem}
The hypothesis $\min\{d_1,e_1\}>\min\{d_2,e_2\}$ is essential for our construction. Indeed if we consider $I_{X_1}:=(x^2,y^2)$ and $I_{X_2}:=(t^2,(x+y+t)^2)$ as ideals in $k[x,y,t],$ our construction produces the ideal $I_{X_1}\cap I_{X_2'}$ where $I_{X_2'}:=(t^2,xy+xt+yt)$ which is different from $I_{X_1}\cap I_{X_2},$ although $I_{X_1}+I_{X_2}=I_{X_1}+I_{X_2'}.$
\end{rem}

\begin{cor}
With the same hypotheses of the Theorem \ref{matr} the ideal $I_{X_1}\cap I_{X_2}$ admits the following free graded resolution (not necessarily minimal)
\begin{multline}\label{ris}
 0\to R(-(d_2+e_2))\oplus\bigoplus_{i=1}^nR(-(d_1+e_1+d_2-\pi_i))\stackrel{M}{\longrightarrow} \\ \to R(-(d_1+d_2))\oplus R(-(e_1+d_2))\oplus\bigoplus_{i=1}^nR(-(e_2+\pi_i))\to I_{X_1}\cap I_{X_2}\to 0.
\end{multline}
\end{cor}
\begin{proof}
The result follows by the degree matrix of $M.$
\end{proof}
The aim of the next section is to apply the previous results to get information on Hilbert functions and graded Betti numbers of $I_{X_1}\cap I_{X_2}.$

\section{Applications to Hilbert functions and graded Betti numbers} 
\markboth{\it Applications to Hilbert functions and graded Betti numbers}
{\it Applications to Hilbert functions and graded Betti numbers}
From the Theorem \ref{matr} we are able to describe all possible Hilbert functions for aCM schemes which are union of two complete intersection schemes of codimension $2$ without common components.
\par 
In the sequel we will set $(a)_+:=\max\{0,a\}.$
\begin{prp}
Let $X_1,X_2\subset\pp^r,$ $r\ge 2,$ be two complete intersection schemes of codimension $2,$ without common components, of type, respectively, $(d_1,e_1)$ and $(d_2,e_2),$ with $e_2=\min\{d_1,e_1,d_2,e_2\}.$
Then the Hilbert function of the Artinian reduction $A$ of $I_{X_1}\cap I_{X_2}$ is
\begin{multline}\label{hf}
H_A(t)=t+1-\sum_{i=1}^n(t+1-e_2-\pi_i)_{\phantom{}_+} - (t+1-d_1-d_2)_{\phantom{}_+} - (t+1-e_1-d_2)_{\phantom{}_+} +\\ \sum_{i=1}^n(t+1-d_1-e_1-d_2+\pi_i)_{\phantom{}_+} + (t+1-d_2-e_2)_{\phantom{}_+}
\end{multline}
where the $\pi_i$'s are the minimal generators degrees of an aG scheme linked to $X_1\cap X_2$ in a complete intersection of type $(e_1,d_1,d_2).$
\end{prp} 
\begin{proof}
If $\min\{d_1,e_1\}>\min\{d_2,e_2\}$ then this result follows immediately by Corollary \ref{ris}. Otherwise the construction of Theorem \ref{matr} produces the scheme $X_1\cup X_2',$ (where $X_2'$ is a complete intersection of the same type of $X_2$) such that $I_{X_1}+I_{X_2'}=I_{X_1}+I_{X_2}.$ So, using sequence (\ref{vietoris}), we get that the resolution of $I_{X_1}\cap I_{X_2}$ is up to cancellation the same as the resolution of $I_{X_1}\cap I_{X_2'}$ from which we get the assertion.
\end{proof}

\begin{rem}
By the the formula (\ref{hf}) we deduce some facts about the ideal $I_{X_1}\cap I_{X_2}.$ For this we use the following setting $d_1\ge e_1$ and $\pi_i\le\pi_{i+1}$ for every $i.$ 
\begin{itemize}
	\item[1)]The degree of the first generator is $e_2+\pi_1;$ indeed, since $I_Z\subseteq I_G,$ we have in particular that $\pi_1\le e_1\le d_1$ so $e_2+\pi_1\le e_1+d_2\le d_1+d_2.$ 
	\item[2)]The degree of the second generator is $e_2+\pi_2;$ indeed, by the previous observation $\pi_2\le\max\{e_1,d_2\},$ so $e_2+\pi_2\le e_1+d_2\le d_1+d_2.$
	\item[3)]Let $\sigma:=\max\{t\mid H_A(t)>0\}$ (the socle degree of the Artinian algebra $A$). Note that, since $I_Z\subseteq I_G,$ if we arrange $(e_1,d_1,d_2)$ in a not decreasing way, we have that they are respectively greater than or equal to $\pi_1\le\pi_2\le\pi_3.$ In any case $\sigma\le e_1+d_1+d_2-\pi_1-2,$ since the biggest degree of a minimal syzygy of $A$ is less than or equal to $e_1+d_1+d_2-\pi_1$ because $d_2+e_2\le e_1+d_1+d_2-\pi_1$ (as $e_1\ge\pi_1$). This bound is sharp iff $e_1>\pi_1.$ Whenever $e_1=\pi_1,$ $\sigma\le e_1+d_1+d_2-\pi_2-2,$ since, in this case, the biggest degree of a minimal syzygy of $A$ is
less than or equal to $e_1+d_1+d_2-\pi_2,$ because $d_2+e_2\le e_1+d_1+d_2-\pi_2$ (as $d_1\ge\pi_2$). This bound is sharp iff $e_1=\pi_1$ and
$d_1>\pi_2.$ Whenever $e_1=\pi_1$ and $d_1=\pi_2,$ $\sigma\le\max\{e_1+d_1+d_2-\pi_3-2,d_2+e_2-2\};$ moreover if $\max\{e_1+d_1+d_2-\pi_3-2,d_2+e_2-2\}=e_1+d_1+d_2-\pi_3-2$ then $\sigma=e_1+d_1+d_2-\pi_3-2.$
  \item[4)]Although the complete intersections have the Hilbert function of decreasing type, this is not true anymore, in general, for the unions of two of them, as we will see in the Proposition \ref{gen}.
\end{itemize}
\end{rem}
\begin{prp}\label{gen}
Let $X_1,X_2\subset\pp^r,$ $r\ge 2,$ be two complete intersection schemes of codimension $2$ of type, respectively, $(d_1,e_1)$ and $(d_2,e_2).$ Assume that $X_1\cup X_2$ is aCM. Let $I_{X_1}=(f_1,g_1)$ and $I_{X_2}=(f_2,g_2),$ $\deg f_i=d_i$ and $\deg g_i=e_i.$ Let us suppose that $f_2\in(f_1,g_1,g_2),$ say $f_2=a_1f_1+b_1g_1+b_2g_2.$ Then the Hilbert-Burch matrix of $I_{X_1}\cap I_{X_2}$ is
 $$M=\begin{pmatrix}-g_2&0\\b_1&f_1\\a_1&-g_1\end{pmatrix}.$$
Consequently $I_{X_1}\cap I_{X_2}=(f_2-b_2g_2,g_1g_2,f_1g_2).$
\par
The graded minimal free resolution of $I_{X_1}\cap I_{X_2}$ is
 $$0\to R(-(d_2+e_2)\oplus R(-(e_1+e_2+d_1)\to R(-d_2)\oplus R(-(e_1+e_2))\oplus R(-(d_1+e_2)).$$
\par
The Hilbert function of the Artinian reduction $A$ of $I_{X_1}\cap I_{X_2}$ is
 $$H_{A}(t)=H_{A_1}(t-e_2)+H_{A_2}(t),$$
where $A_i$ is the Artinian reduction of $I_{X_i}.$
\end{prp}
\begin{proof}
Let $I_Y=I_2(M).$ Note that $\hgt I_Y=\hgt(f_2-b_2g_2,g_1g_2,f_1g_2)=2;$ namely, if $\hgt I_Y=1$ the three generators should have a common factor. Since $(f_1,g_1)$ is a regular sequence $f_2$ should have a common factor with $g_2,$ a contradiction.
\par
Trivially $I_{Y}\subseteq I_{X_1}\cap I_{X_2};$ on the other hand an easy computation shows that $\deg I_Y=d_1e_1+d_2e_2=\deg I_{X_1}\cap I_{X_2},$ so $I_Y=I_{X_1}\cap I_{X_2}.$
\par
A minimal set of generators of $I_{X_1}\cap I_{X_2}$ and its resolution can be deduced immediately from the matrix $M.$
\par
With regard to the Hilbert function, we have, for every $t\in\zz,$ that
\begin{multline*}
H_A(t)=(t+1)_{\phantom{}_+}-(t+1-d_2)_{\phantom{}_+} - (t+1-e_1-e_2)_{\phantom{}_+} - (t+1-d_1-e_2)_{\phantom{}_+} +\\ +(t+1-d_2-e_2)_{\phantom{}_+} + (t+1-e_1-e_2-d_1)_{\phantom{}_+};
\end{multline*}
\begin{multline*}
H_{A_1}(t-e_2)=(t+1-e_2)_{\phantom{}_+}-(t+1-e_2-d_1)_{\phantom{}_+}+\\ - (t+1-e_2-e_1)_{\phantom{}_+}+ (t+1-e_2-d_1-e_1)_{\phantom{}_+};
\end{multline*}
\begin{multline*}
H_{A_2}(t)=(t+1){\phantom{}_+}-(t+1-d_2)_{\phantom{}_+} - (t+1-e_2)_{\phantom{}_+}+ (t+1-d_2-e_2)_{\phantom{}_+},
\end{multline*}
from which we get our formula.
\end{proof}

\begin{cor}
Let $X_1,X_2\subset\pp^r,$ $r\ge 2,$ be two complete intersection schemes of codimension $2$ of type, respectively, $(d_1,e_1)$ and $(d_2,e_2).$ Assume that $X_1\cup X_2$ is aCM. Let $I_{X_1}=(f_1,g_1)$ and $I_{X_2}=(f_2,g_2),$ $\deg f_i=d_i$ and $\deg g_i=e_i.$ Let us suppose that $d_2\ge d_1+e_1+e_2-2$ and $(f_1,g_1,g_2)$ is a regular sequence. Then the same conclusions of Proposition \ref{gen} hold.
\par
In particular when $d_2>d_1+e_1+e_2$ then $H_{A}$ is not of decreasing type, where $A$ is the Artinian reduction of $I_{X_1}\cap I_{X_2}.$
\end{cor}
\begin{proof}
Our assumptions imply that $I_{X_1}+I_{X_2}$ is an aCM ideal of height $3.$ Let $B$ be the Artinian reduction of the complete intersection ideal $(f_1,g_1,g_2).$ Then $B_t=0$ for $t\ge d_1+e_1+e_2-2,$ so, since $d_2\ge d_1+e_1+e_2-2,$ $f_2\in(f_1,g_1,g_2);$ now applying Proposition \ref{gen} we get our assertion.
\par
If $d_2>d_1+e_1+e_2$ then $H_A(t)=H_{A_1}(t-e_2)+H_{A_2}(t)=e_2<e_1+e_2$ for every $t$ such that $d_1+e_1+e_2-1\le t\le e_2-1,$ i.e. $H_A$ takes the same value (less than $e_1+e_2$) in at least two adjacent degrees.
\end{proof}
In the next proposition we collect results on the graded Betti numbers of our schemes which are consequences of Theorem \ref{matr}. 
\begin{prp}\label{betti}
With the same assumptions of Theorem \ref{matr} we have
\begin{itemize}
	\item[i.]The graded Betti numbers can be obtained by the resolution (\ref{ris}) just deleting at most three terms in degrees $d_1+d_2,$ $e_1+d_2,$ $e_2+d_2.$
	\item[ii.]In any case two among the products $f_1f_2,$ $f_1g_2,$ $g_1f_2,$ $g_1g_2$ are minimal generators for $I_{X_1}\cap I_{X_2}.$
\end{itemize}
\end{prp}
\begin{proof}
 \item[i.]It is enough to observe that the only units in the matrix $M$ can appear in the first two rows or in the first column.
 \item[ii.]If the resolution (\ref{ris}) is minimal, then $f_1f_2$ and $g_1f_2$ are the first two maximal minors of the matrix $M.$ 
Otherwise let us suppose that, say $f_1f_2$ is not a minimal generator for $I_{X_1}\cap I_{X_2}.$ This implies that $e_1=\pi_i$ for some $i$ and $g_1$ is a minimal generators for $I_G,$ so we can replace $p_i$ with $g_1.$ Note that $\deg\det M_{[i+2;-]}=\pi_i+e_2=e_1+e_2,$ so $\det M_{[i+2;-]}$ can be choose as a minimal generator for $I_{X_1}\cap I_{X_2}.$ Now $\det M_{[i+2;-]}=q+g_1\pf\overline{A}$ (see the computation in the proof of Theorem \ref{matr}). But $q+g_1\pf\overline{A}= q+g_1g_2+\lambda f_2g_1$ for some $q$ and $\lambda,$ then $g_1g_2$ can replace it as a minimal generator for $I_{X_1}\cap I_{X_2}.$ Analogously when $g_1f_2$ is not a minimal generator for our ideal the same argument shows that we can take $f_1g_2$ as a minimal generator.
\end{proof}

\begin{exm}
We produce an example in which the resolution (\ref{ris}) is a minimal free resolution for $I_{X_1}\cap I_{X_2}.$ In $R=k[x_0,..x_8],$ let us consider the following two complete intersections $I_{X_1}=(f_1,g_1)$ e $I_{X_2}=(f_2,g_2)$ where
 $$f_1=x_0^3x_1^3x_7,\,\,\,g_1=x_2^3x_5^3x_6$$
$$f_2=(x_0^3+x_2^3+x_4^3)x_3^3x_8+(x_0^3x_7-x_5^3x_8)x_1^3,\,\,\,g_2=(x_0^3+x_2^3+x_4^3)x_6x_7x_8.$$
Let $I_G:=(f_1,g_1,f_2):(f_1,g_1,f_2,g_2);$ $I_G$ is a Gorenstein ideal with $5$ generators in degree $6.$ Consequently the resolution (\ref{ris}) is
$$0\to R(-13)\oplus R(-15)^5\to R(-12)^5\oplus R(-14)^2\to I_{X_1}\cap I_{X_2}\to 0,$$
which is clearly minimal.
\end{exm}

\begin{rem}
Note that concerning to the cancellations we describe in Proposition \ref{betti} all the possibilities could happen. Indeed, it will be enough to choose for a suitable Gorenstein $G$ a complete intersection containing it and whose generators are or are not minimal generators for $G.$
\end{rem}

\begin{rem}
Note that in the resolution \ref{ris} a syzygy of degree $d_2+e_2$ appears. It induces via the map in the exact sequence \ref{vietoris} the trivial syzygy on $I_{X_2}.$ This implies that we will have a cancellation in degree $d_2+e_2$ in the mapping cone in the exact sequence \ref{vietoris}. In fact we have $I_{X_1}\cap I_{X_2}=(f_1f_2,g_1f_2,h_1,\ldots,h_n),$  where $h_i:=\det M_{[i+2;-]}$ (see the proof of Theorem \ref{matr}). The syzygy of degree $d_2+e_2,$ in the same notation of Theorem \ref{matr}, is $(0,0,\gamma_1,\ldots,\gamma_n).$ From the proof of Theorem \ref{matr} we have that $h_i=\lambda_if_2+p_i\pf\overline{A},$ so we get
 $$\sum_{i=1}^n\gamma_i(\lambda_if_2+p_i\pf\overline{A})=0\rw
 \sum_{i=1}^n\gamma_i\lambda_if_2+\sum_{i=1}^n\gamma_ip_i\pf\overline{A}=0,$$
since $\sum_{i=1}^n\gamma_ip_i=f_2$ we are done.
\end{rem}

\vspace{1cm}
{\f
{\sc (A. Ragusa) Dip. di Matematica e Informatica, Universit\`a di Catania,\\
                  Viale A. Doria 6, 95125 Catania, Italy}\par
{\it E-mail address: }{\tt ragusa@dmi.unict.it} \par
{\it Fax number: }{\f +39095330094} \par
\vspace{.3cm}
{\sc (G. Zappal\`a) Dip. di Matematica e Informatica, Universit\`a di Catania,\\
                  Viale A. Doria 6, 95125 Catania, Italy}\par
{\it E-mail address: }{\tt zappalag@dmi.unict.it} \par
{\it Fax number: }{\f +39095330094}
}

\end{document}